\documentclass[12pt]{amsart}

\usepackage{fullpage,url,amssymb,amsmath,amsthm,amsfonts,mathrsfs,enumitem,comment,latexsym}
\usepackage[usenames,dvipsnames]{color}
\usepackage[pagebackref = true, colorlinks = true, linkcolor = blue, citecolor = Green]{hyperref}
\usepackage[alphabetic,lite]{amsrefs}
\usepackage[all, cmtip]{xy} 
\usepackage{amscd}   
\usepackage{rotating}
\usepackage{graphicx}

\DeclareFontEncoding{OT2}{}{} 
\newcommand{\textcyr}[1]{%
 {\fontencoding{OT2}\fontfamily{wncyr}\fontseries{m}\fontshape{n}\selectfont #1}}

\newcommand{\defi}[1]{\textsf{#1}} 
\def\act#1#2%
  {\mathop{}%
   \mathopen{\vphantom{#2}}^{#1}%
   \kern-3\scriptspace%
   #2}
   
\usepackage{color} 

\newcommand{\Z}{{\mathbb Z}}
\newcommand{\Q}{{\mathbb Q}}

\newcommand{\F}{{\mathbb F}}
\newcommand{\A}{{\mathbb A}}

\newcommand{\G}{{\mathbb G}}

\newcommand{\ksep}{{k_s}}

\newcommand{\calA}{{\mathcal A}}

\newcommand{\Sha}{{\mbox{\textcyr{Sh}}}}

\newcommand{\To}{\longrightarrow}

\DeclareMathOperator{\Gal}{Gal}

\DeclareMathOperator{\HH}{H}

\DeclareMathOperator{\Aut}{Aut}

\DeclareMathOperator{\res}{res}

\newtheorem{Theorem}{Theorem}
\newtheorem{Lemma}[Theorem]{Lemma}
\newtheorem{Proposition}[Theorem]{Proposition}
\newtheorem{Corollary}[Theorem]{Corollary}

\newtheorem{Remark}[Theorem]{Remark}

\numberwithin{equation}{section}

\begin{document}

\title{Local-global principles for Weil-Ch\^atelet divisibility in positive characteristic}
\author{Brendan Creutz}
\address{School of Mathematics and Statistics, University of Canterbury, Private Bag 4800, Christchurch 8140, New Zealand}
\email{brendan.creutz@canterbury.ac.nz}
\urladdr{http://www.math.canterbury.ac.nz/\~{}bcreutz}
\author{Jos\'e Felipe Voloch}
\address{School of Mathematics and Statistics, University of Canterbury, Private Bag 4800, Christchurch 8140, New Zealand and Department of Mathematics, University of Texas, Austin, TX 78712, USA}
\email{felipe.voloch@canterbury.ac.nz}
\urladdr{http://www.math.canterbury.ac.nz/\~{}f.voloch}

\begin{abstract}
We extend existing results characterizing Weil-Ch\^atelet divisibility of locally trivial torsors over number fields to global fields of positive characteristic. Building on work of Gonz{\'a}lez-Avil{\'e}s and Tan, we characterize when local-global divisibility holds in such contexts, providing examples showing that these results are optimal. We give an example of an elliptic curve over a global field of characteristic $2$ containing a rational point which is locally divisible by $8$, but is not divisible by $8$ as well as examples showing that the analogous local-global principle for divisibility in the Weil-Ch\^atelet group can also fail. 
\end{abstract}

\maketitle

\section{Introduction}

Let $k$ be a global field and $A/k$ an abelian variety. The flat cohomology group $\HH^1(k,A)$ is called the Weil-Ch\^atelet group. It parameterizes $k$-torsors under $A$ and contains the Shafarevich-Tate group,
\[
	\Sha(k,A) := \ker\left(\HH^1(k,A) \to \prod \HH^1(k_v,A)\right)\,,
\]
where the product runs over all completions $k_v$ of $k$. It is conjectured that $\Sha(k,A)$ is finite, and in particular that it contains no nontrivial divisible elements. Cassels asked whether the elements of $\Sha(k,A)$ are divisible in the larger group $\HH^1(k,A)$  \cite[Problem 1.3]{CasselsIII}. Closely related to this is the question of whether, for given integers $r \ge 0$ and $m$, the map
\begin{equation}\label{eq:1}
	\HH^r(k,A)/m\HH^r(k,A) \to \prod \HH^r(k_v,A)/m\HH^r(k_v,A)
\end{equation}
is injective. Indeed, a positive answer to Cassels' question follows from the injectivity of these maps for $r = 1$ and $m \ge 1$. In the case that the characteristic $p$ of $k$ does not divide $m$ this has been investigated in \cite{Bashmakov,CasselsIV,CipStix,CreutzWCDiv,CreutzWCDiv2,DZ1,DZexamples,DZ2,LawsonWuthrich,PRV-2,PRV-3}. In particular, when $p \nmid m$ it is known that the \defi{local-global principle for divisibility by $m$ in $\HH^r(k,A)$ holds} (i.e., that the map in~\eqref{eq:1} is injective) in all of the following cases:
	\begin{enumerate}
		\item\label{r = 2} $r \ge 2$ (\cite[Theorem 2.1]{CreutzWCDiv2});
		\item\label{Tate'sResult} $A$ is an elliptic curve, and $m$ is prime (\cite[Lemma 6.1]{CasselsIV});
		\item\label{ell5} $A$ is an elliptic curve over $\Q$ and $m = \ell^n$ is a power of a prime $\ell \ge 5$ \cite[]{CipStix} \cite[Corollary 4]{PRV-3} and \cite[Theorem 24]{LawsonWuthrich};
		\item\label{ellbig} $A$ is an elliptic curve over a number field $k$ and $m = \ell^n$ is a power of a sufficiently large prime, depending only on the degree of $k$ \cite{CipStix,PRV-2}.
	\end{enumerate}
	
	On the other hand, there are examples showing that the local-global principle for divisibility by $m$ in $\HH^r(k,A)$ can fail in the following situations:
	\begin{enumerate}[resume]
		\item\label{m=2} $r = 0$, $A$ is an abelian surface over $\Q$ and $m = 2$ \cite[p. 61]{CasselsFlynn};
		\item $r = 0$, $A$ is an elliptic curve over $\Q$ and $m = 2^n$ with $n \ge 2$ \cite{DZ2};
		\item\label{it:7} $r \in \{0,1\}$, $A$ is an abelian variety over $\Q$ and $m$ is any prime number \cite{CreutzWCDiv};
		\item\label{it:8} $r \in \{0,1\}$, $A$ is an elliptic curve over $\Q$ and $m = 3^n$ with $n \ge 2$ \cite{CreutzWCDiv2}.
	\end{enumerate}
	Moreover, in~\eqref{it:7} and~\eqref{it:8} the examples given for $r = 1$ satisfy $\Sha(k,A) \not\subset m\HH^1(k,A)$ showing that the answer to Cassels' original question is also no.

	In this paper we are interested in these questions when $m = p^n$ is a power of the characteristic of $k$. In this case \cite[Proposition 3.5]{GA-THassePrinciple} shows that for an elliptic curve $A/k$, the local-global principle for divisibility by $p^n$ in $\HH^r(k,A)$ holds except possibly if $r \in \{0,1\}$, $p = 2$ and $n \ge 3$ (for $r \ge 2$ see Lemma~\ref{lem:1}). Their results imply that the only possible counterexamples must be non-constant elliptic curves with $j$-invariant in $k^8$ (see Corollary \ref{cor:jinv} below). We show that this is sharp and, moreover, that Cassels' question has a negative answer over global fields of positive characteristic. Specifically we prove the following.
	
	\begin{Theorem}[Proposition~\ref{prop:nonisoH1example}]\label{thm:H1example}
		There exists a non-isotrivial elliptic curve $E$ over $k = \F_2(t)$ with $j(E) \in k^8$ such that $\Sha(k,E) \not\subset 8\HH^1(k,E)$.
	\end{Theorem}
	
	\begin{Theorem}[Proposition~\ref{prop:nonisoH0example}]\label{thm:H0example}
		There exists a non-isotrivial elliptic curve $E$ over $k = \F_2(t)$ with $j(E) \in k^8$ such that the local-global principle for divisibility by $2^n$ in $E(k)$ fails for every $n \ge 3$.
	\end{Theorem}

	In Propositions~\ref{prop:isoH1example} and~\ref{prop:isoH0example} we give examples of isotrivial elliptic curves over global fields of characteristic $2$ for which the local-global principle for divisibility by $8$ can fail. Interestingly, such examples do not occur over $\F_2(t)$ (see Proposition~\ref{prop:iso}).
	
	\begin{Remark}
		Although we have not pursued it here, it would also be interesting to see if there are analogues of~\eqref{ell5} and~\eqref{ellbig} above over function fields of curves over finite fields of positive characteristic prime to $\ell$ in terms of, say, genus or gonality. One could then also ask for examples along the lines of~\eqref{m=2}--\eqref{it:8}.
	\end{Remark}

	All of the results above establishing the local-global principle for divisibility by $m$ in $\HH^0(k,A)$ and $\HH^1(k,A)$ are obtained by showing that the group $\Sha^1(k,A[m])$ (or its dual) of locally trivial classes in $\HH^1(k,A[m])$ vanishes. Building on \cite{GA-THassePrinciple}, we give a necessary and sufficient criterion for the vanishing of $\Sha^1(k,A[m])$ in the case that $A$ is an elliptic curve and $m = p^n$ is a power of the characteristic of $k$ (see Theorem~\ref{thm:computeSha1}). This allows us to construct the examples in Theorems~\ref{thm:H1example} and~\ref{thm:H0example} demonstrating the failure of the local-global principle for divisibility in $\HH^r(k,A)$. To obtain the stronger result of Theorem~\ref{thm:H1example} that $\Sha(k,A) \not\subset 8\HH^1(k,A)$) we establish the following characterization of  the divisibility of $\Sha(k,A)$ in the Weil-Ch\^atelet group.
	
	\begin{Theorem}\label{thm:characterizedivisibility}
		Let $A$ be an abelian variety over a global field $k$ of characteristic $p$ with dual abelian variety $A^*$ and let $m = p^n$. An element $T \in \Sha(k,A)$ lies in $m\HH^1(k,A)$ if and only if the image of the map 
		\[
			\Sha^1(k,A^*[m]) \to \Sha(k,A^*)
		\]
		induced by the inclusion of group schemes $A^*[m] \subset A^*$ is orthogonal to $T$ with respect to the Cassels-Tate pairing. In particular, $\Sha(k,A) \subset m\HH^1(k,A)$ if and only if the image of $\Sha^1(k,A^*[m])$ lies in the divisible subgroup of $\Sha(k,A^*)$.	
	\end{Theorem}
	
		The proof of this theorem is given in Section~\ref{sec:ProofOfTheorem3}. The theorem also holds when $m$ is not divisible by the characteristic; see \cite[Theorem 4]{CreutzWCDiv} where a proof using the ``Weil pairing definition" of the Cassels-Tate pairing from \cite{PoonenStoll} is given. To handle the case that $m = p^n$ we make use of duality theorems in flat cohomology developed in \cite[Chapter 3]{MilneADT} and \cite{G-Aduality}.

	\section{Locally trivial torsors and divisibility}

	The orthogonality condition in Theorem~\ref{thm:characterizedivisibility} holds trivially if $\Sha^1(k,A^*[m]) = 0$. When this is the case more is true.
	
		\begin{Lemma}\label{lem:1}
			Maintain the notation from Theorem~\ref{thm:characterizedivisibility}. Assume any of the following:
			\begin{enumerate}
				\item $r = 0$ and $\Sha^1(k,A[m]) = 0$,
				\item $r = 1$ and $\Sha^1(k,A^*[m]) = 0$, or
				\item $r \ge 2$.
			\end{enumerate}
			Then the local-global principle for divisibility by $m$ holds in $\HH^r(k,A)$.
		\end{Lemma}
		
		\begin{Remark}
			Like Theorem~\ref{thm:characterizedivisibility}, this is also true when $m$ is not divisible by the characteristic of $k$ \cite[Theorem 2.1]{CreutzWCDiv2}; the proofs of statements (1) and (2) given below are valid without assumption on the characteristic.
			\end{Remark}
		
		\begin{proof}
			Recall that $m = p^n$ is a power of the characteristic of $k$. The sequence
			\begin{equation}\label{eq:KummerSeq}
				0 \to A[p^n] \to A \stackrel{p^n}\to A \to 0
			\end{equation}	
			is exact on the flat site and gives rise to an exact sequence in flat cohomology,
			\begin{equation*}\label{eq:KummerSeqBoundaries}
				\HH^r(k,A) \stackrel{(p^n)_*} \to \HH^r(k,A) \to \HH^{r+1}(k,A[p^n])\,.
			\end{equation*}	
			From this we see that the obstruction to an element of $\HH^r(k,A)$ being divisible by $p^n$ is a class in $\HH^{r+1}(k,A[p^n])$. Therefore an element of $\HH^r(k,A)$ that is locally, but not globally, divisible by $p^n$ must give a nontrivial element of $\Sha^{r+1}(k,A[p^n])$. In case (1) the conclusion follows immediately from this observation.
			
			In case (2) the conclusion follows from this observation once one takes into account that there is a perfect pairing
			\[
				\Sha^1(k,A^*[p^n]) \times \Sha^2(k,A[p^n]) \to \Q_p/\Z_p\,
			\]
			\cite[Theorem 1.1]{G-Aduality}.
			
			 In case (3) the statement holds trivially because $\HH^r(k,A)(p) = 0$. Indeed for $r = 2$ this is \cite[Lemma 3.3]{GA-THassePrinciple} and for $r \ge 3$ this follows from the fact that $k$ has strict $p$-cohomological dimension $2$ (\cite[Prop. 6.1.9 and Prop. 6.1.2]{GilleSzamuely}) and the fact that $\HH^r(k,A) = \HH^r(\Gal(k),A(\ksep))$ since $A$ is smooth.
		\end{proof}

	In light of Theorem~\ref{thm:characterizedivisibility} and Lemma~\ref{lem:1} we are interested in determining the group $\Sha^1(k,A[m])$. Using results of \cite{GA-THassePrinciple} and a well known argument using Chebotarev's density theorem (e.g., \cite[Proposition 8]{Serre1964}) this may be reduced to a computation in the cohomology of finite groups. In the case that $A[m](\ksep)$ is cyclic, the required computation is a classical one used in the proof of the Grunwald-Wang Theorem. In this way, we obtain the following theorem.
	
	\begin{Theorem}\label{thm:computeSha1}
		Suppose $A$ is an abelian variety over a global field $k$ of characteristic $p$ with separable closure $\ksep$ such that $A[p^n](\ksep)$ is cyclic (this is the case, for example, if $A$ is an elliptic curve). Set $K = k(A[p^n](\ksep))$ and $G = \Gal(K/k)$. Then $\Sha^1(k,A[p^n]) = 0$, unless $G$ is noncyclic and not isomorphic to any of its decomposition groups, in which case $\Sha^1(k,A[p^n]) = \Z/2\Z$.
	\end{Theorem}
	
	 \begin{Remark}\label{rem:oneplace}
		When $E/k$ is an ordinary elliptic curve with $j(E) \notin k^p$, a stronger statement is true: the restriction map $\HH^1(k,E[p^n]) \to \HH^1(k_v,E[p^n])$ is injective for any prime $v$ of $k$. In the case $n = 1$ this follows from the existence of an injective map $\HH^1(k,E[p]) \hookrightarrow k$ (functorial in $k$) (see \cite{Ulmer_pdesc,Voloch_pdesc}). The general case is proved by the following induction argument suggested to us by D. Ulmer. The assumption $j(E) \notin k^p$ implies that $\HH^0(k,E[p^{n-1}]) = 0$, so from the flat cohomology of the exact sequence $0 \to E[p] \to E[p^n] \to E[p^{n-1}] \to 0\,$ we obtain a commutative diagram with exact rows,
			\[
				\xymatrix{
					0 \ar[r]& \HH^1(k,E[p]) \ar[r]\ar@{^{(}->}[d] &\HH^1(k,E[p^n]) \ar[r]\ar[d] &\HH^1(k,E[p^{n-1}])\ar[d] \\
					0 \ar[r]& \HH^1(k_v,E[p]) \ar[r]&\HH^1(k_v,E[p^n]) \ar[r] &\HH^1(k_v,E[p^{n-1}])
					}
			\]
			The vertical map on the right is injective by the induction hypothesis, so the vertical map in the middle must be as well.
	\end{Remark}
	

	\begin{proof}[Proof of Theorem~\ref{thm:computeSha1}]
		To ease notation, set $M = A[p^n]$. Then $K = k(M(\ksep))$ and $G = \Gal(K/k)$. By the Main Theorem of \cite{GA-THassePrinciple} (or \cite[Example C.4.3]{CGP}), $\Sha^1(k,M) \simeq \Sha^1(\Gal(k),M(\ksep))\,.$ The inflation map gives an exact sequence
		\[
			0 \to \HH^1(G,M(\ksep)) \to \HH^1(\Gal(k),M(\ksep)) \to \HH^1(\Gal(K),M(\ksep))\,.
		\]	
		The action of $\Gal(K)$ on $M(\ksep)$ is trivial, so $\Sha^1(\Gal(K),M(\ksep)) = 0$ by Chebotarev's theorem. It follows that $\Sha^1(\Gal(k),M(\ksep))$ is isomorphic to
		\[
			\Sha^1(G,M) := \ker\left(\HH^1(G,M(\ksep)) \stackrel{\res_v}{\To} \prod_v \HH^1(G_v,M(\ksep))\right)\,,
		\]
		where $G_v \subset G$ denotes the decomposition group at the prime $v$ and the product runs over all primes. Since every cyclic subgroup of $G$ occurs as a decomposition group, Theorem \ref{thm:computeSha1} follows from the next lemma.
	\end{proof}
	
		\begin{Lemma}\label{lem:grcohom}
			Let $G \subset (\Z/p^N\Z)^\times = \Aut(\Z/p^N\Z)$ and let $M$ be the $G$-module isomorphic to $\Z/p^N\Z$ on which $G$ acts in the canonical way.
			\begin{enumerate}
				\item\label{it:1} If $G$ is not cyclic, then $p = 2$, $N \ge 3$ and $-1 \in G$.
				\item\label{it:2} $\HH^1(G,M) = 0$ unless $p = 2$, $N \ge 3$ and $-1 \in G$, in which case $\HH^1(G,M) \simeq \Z/2\Z\,.$
				\item\label{it:3} If $G' \subsetneq G$ is a proper subgroup, then the restriction map $\HH^1(G,M) \to \HH^1(G',M)$ is the zero map.
			\end{enumerate}
		\end{Lemma}

		\begin{proof}
			\eqref{it:1} is easy and \eqref{it:2} is a well known computation (c.f. \cite[Lemma 9.1.4]{CNF}). To prove \eqref{it:3} we may assume that $-1 \in G'$. Then we can write $G = \mu_2 \times \langle \alpha \rangle$, and $G' = \mu_2 \times \langle \alpha^{k} \rangle$, where $k = [G:G'] \ge 2$ and $\alpha \in 1 + 2^su$ with $s \ge 2$ and $u$ odd. 			
			 Consider cohomology of the short exact of $\mu_2$-modules,
			\[
				0 \to M^{\langle \alpha \rangle} \stackrel{i}\to M^{\langle \alpha^{k}\rangle} \to Q \to 0,
			\]
			where $Q$ is the quotient. Since $M^{G} = M^{G'} = M^{\mu_2}$, this gives an exact sequence
			\[
				0 \to Q^{\mu_2} \hookrightarrow \HH^1(\mu_2,M^{\langle \alpha \rangle}) \stackrel{i_*}\to \HH^1(\mu_2,M^{\langle\alpha^{k}\rangle})
			\]
			We claim that all terms in this sequence have order $2$, so $i_* = 0$. Indeed, $Q^{\mu_2} \simeq \Z/2\Z$, as $Q$ is cyclic and contains an element of order $2$ (since $2 \mid k$). The proof of \eqref{it:2} shows that the inflation map gives an isomorphism $\Z/2\Z \simeq \HH^1(\mu_2,M^{\langle \alpha\rangle}) \simeq \HH^1(G,M)$, and similarly for $G'$. Under these isomorphisms, the restriction map $\HH^1(G,M) \to \HH^1(G',M)$ is given by $i_*$ which is the zero map. 
		\end{proof}
		
		\begin{Corollary}\label{cor:jinv}
			Suppose $E$ is an elliptic curve over a global field $k$ of characteristic $p$. Then the local-global principle for divisibility by $p^n$ holds in $\HH^r(k,E)$ except possibly if $r \in \{0,1\}$, $p = 2$, $n \ge 3$, $E$ is a non-constant ordinary elliptic curve and the $j$-invariant of $E$ is an $8$-th power. 
		\end{Corollary}
		
		\begin{proof}
			By Lemma~\ref{lem:1}, the local-global principle for divisibility by $p^n$ can only fail for $r \in \{0,1\}$ and then only when $\Sha^1(k,E[p^n]) \ne 0$. Let $K = k(E[p^n](\ksep))$, $G = \Gal(K/k)$, and $N$ be such that $E[p^n](\ksep) \simeq \Z/p^N\Z$. Then $N \le n$ and $j(E) \in k^{p^N}$ by \cite{KatzMazur} Proposition 12.2.7. If $\Sha^1(k,E[p^n]) \ne 0$, then $G$ is not cyclic, by Theorem~\ref{thm:computeSha1}, and $p = 2$ and $3 \le N$, by Lemma~\ref{lem:grcohom}. This implies that $E$ is ordinary, since otherwise $N = 0$. Finally, if $E$ is constant, then the $p^n$-torsion points are defined over a finite field and $G$ is cyclic, a contradiction. 
		\end{proof}
		
		For supersingular elliptic curves we have the following.
		\begin{Lemma}
			Suppose $E$ is a supersingular elliptic curve over a global field $k$ of characteristic $p$. Then $\HH^1(k,E)$ is $p$-divisible.
		\end{Lemma}
	
	\begin{proof}
		The $p$-torsion subgroup-scheme of $E$ sits in an exact sequence
		\begin{equation}\label{eq:supsing}
			0 \to \alpha_p \to E[p] \to \alpha_p \to 0\,,
		\end{equation}
		where $\alpha_p$ is the kernel of $F:\G_a \to \G_a$, defined by $F(x) = x^p$. Since $\HH^i(k,\G_a) = 0$ for $i > 0$, we see that $\HH^i(k,\alpha_p) = 0$ for $i \ne 1$. The long exact sequence associated to \eqref{eq:KummerSeq} then shows that $\HH^2(k,E[p]) = 0$. So multiplication by $p$ on $\HH^1(k,E)$ is surjective.	
	\end{proof}
		\begin{Remark}
The lemma also holds for supersingular abelian varieties of arbitrary dimension.
Indeed, in this case, the $p$-torsion subgroup-scheme admits a filtration with successive
quotients isomorphic to $\alpha_p$ and a similar argument applies.
		\end{Remark}
	
		In the remainder of this section we collect various results that will be used to construct the examples in the following section.
		
		\begin{Lemma}\label{lem:torsion}
			Suppose $E$ is an elliptic curve over a global field $k$ of characteristic $p$ and that $P \in E(k)$ is locally divisible by $m = p^n$, but not globally divisible by $m$. Then the image of $P$ in $E(k)/mE(k)$ does not lie in the image of $E(k)_{tors}$.
		\end{Lemma}
				
		\begin{proof}
			The assumption on $P$ implies that its image under the connecting homomorphism $\delta:E(k)/mE(k) \to \HH^1(k,E[m])$ is a nontrivial element of $\Sha^1(k,E[m])$. By Theorem~\ref{thm:computeSha1} we may assume that $m = 2^n$ and that there is some completion of $k$ such that $E(k_v)$ contains no point of exact order $2^n$. Now by way of contradiction suppose $P = 2^nQ + T$ with $Q\in E(k)$ and $T \in E(k)_{tors}$. The only nontrivial element of $\Sha^1(k,E[m])$ has order $2$, so $T \in E[2]$. Since $P$ is locally divisible by $2^n$, we can find $R \in E(k_v)$ such that $P = 2^nR$. But then $T = 2^n(R-Q)$ which shows that $R-Q$ is a point of order $2^{n+1}$ in  $E(k_v)$
		\end{proof}

		\begin{Proposition}\label{prop:L/k}
			Let $k$ be a global field of characteristic $p$ and suppose $\Sha^1(k,A[p^n]) \ne 0$. Then there exists a separable extension $L/k$ and a point $P \in A(L)$ that is locally divisible by $p^n$, but not globally divisible by $p^n$.
		\end{Proposition}
		
		\begin{Remark}
			Our proof is based on that of \cite[Theorem 3]{DZ2} which proves the result when $p$ does not divide the characteristic of $k$.
		\end{Remark}
		
		\begin{proof}
			To ease notation, set $M = A[p^n]$. Let $\xi \in \Sha^1(k,M)$ be a nonzero element. By the Main Theorem of \cite{GA-THassePrinciple} (see also \cite[Example C.4.3]{CGP}), $\xi$ is in the image of the inflation map $\HH^1(\Gal(k),M(\ksep))$. In particular, a separable extension $L/k$ kills $\xi$ if and only if $M(L) \ne \emptyset$. The class $\xi \in \HH^1(k,M)$ may be represented by an $A$-torsor $\pi: T \to A$ under $M$,  (also called an $M$-covering of $A$). For any separable extension $L/k$, consider the exact sequence, 
			\[
				A(L)/p^n \stackrel{\delta} \hookrightarrow \HH^1(L,A[p^n]) \to \HH^1(L,A) = \HH^1(\Gal(L),A(\ksep))\,,
			\]
			where the equality follows from the fact that $A$ is smooth. The image of $\xi$ in $\HH^1(L,A)$ is trivial if and only if $T(L) \ne \emptyset$. When this is the case, $\res_L(\xi) = \delta(P)$ for some $P \in A(L)$. Therefore, to prove the proposition it suffices to find a separable extension $L/k$ linearly disjoint from $K := k(M(\ksep))$ such that $T(L) \ne \emptyset$.
			
			Note that $T$ is geometrically irreducible, since $T_\ksep$ is isomorphic to $A_\ksep$. Therefore, $T$ is $k$-birational to a hypersurface $H : f(x_1,\dots,x_s,y) = 0$, where $f \in k[x_1,\dots,x_s,y]$ is irreducible in $\ksep[x_1,\dots,x_s,y]$ and separable in $y$. This follows from the fact that there is a separating basis for the function field of $T$ over $k$ \cite[Lemma 2.6.1]{FriedJarden}. Since global fields are Hilbertian \cite[Thm 13.4.2]{FriedJarden}, there exist a Zariski dense set of ${\bf a} = (a_1,\dots,a_s) \in \A^s(k)$ such that $f({\bf a},y)$ is irreducible over $K$ \cite[Corollary 12.2.3]{FriedJarden}. Adjoining a root of $f({\bf a},y)$ to $k$ gives rise to a separable extension $L/k$ such that $H(L) \ne \emptyset$ and $M(L) = \emptyset$. Since the set of such ${\bf a}$ is Zariski dense, we conclude that the same is true with $T$ in place of $H$.
		\end{proof}

	\begin{Lemma}\label{lem:8tors}
		Suppose $k$ is a field of characteristic $2$ and $E/k: y^2+xy = x^3+ax^2+b$. Then the field of definition of the $8$-torsion points of $E$ is $k(b^{1/8},u,v)$, where $u^2+u = a, v^2 + v =b$.
	\end{Lemma}
	
	\begin{proof} Let $K$ be the field of definition of the $8$-torsion points of $E$. Then the $j$-invariant of $E$ is an $8$-power in $K$ by
\cite{KatzMazur} Proposition 12.2.7. On the other hand $j(E)=1/b$, hence
	$k(b^{1/8}) \subset K$.
	Now, the non-trivial $2$-torsion point is $(0,b^{1/2})$. The
	$4$-torsion is generated by $(b^{1/4},b^{1/2}+ub^{1/4})$. 
	So $K$ contains $k(b^{1/4},u)$. Finally, to get the $8$-torsion one needs to solve
	$w^2+w = b^{1/4} + a$ in $k(b^{1/4},u)$ by the 
	formulas in e.g. \cite{Voloch_pdesc}, which is equivalent to solving 
	$v^2 + v =b$.
	\end{proof}

	\begin{Theorem}[Milne, Tate]
        \label{thm:BSD}
		Let $E$ be an elliptic curve over a global field $k$ of positive characteristic. Then
		\begin{enumerate}
			\item The rank of $E(k)$ is at most the analytic rank of $E$.
			\item If $E$ is isotrivial or if the analytic and algebraic ranks of $E$ coincide, then
			\begin{enumerate}
				\item $\Sha(k,E)$ is finite and has the order predicted by the Birch and Swinnerton-Dyer conjectural formula.
				\item The Cassels-Tate pairing on $\Sha(k,E)$ is nondegenerate.
			 \end{enumerate} 
		\end{enumerate} 
	\end{Theorem}
	
See Tate \cite{Tate-BSD} and Milne \cite{Milne1975}, Theorem 8.1.
The restriction that $p \ne 2$ in \cite{Milne1975} can lifted thanks to \cite{Illusie1979}.

	\section{Examples}\label{sec:examples}
	
	In this section we give examples of elliptic curve over global fields of characteristic $2$ for which the local-global principle for divisibility by powers of $2$ fails.	
	
	\begin{Proposition}\label{prop:nonisoH1example}
			For the elliptic curve $E : y^2 + xy = x^3 + t^8x^2 + (t^{16} + 1)/t^8$ over $k = \F_2(t)$ we have $\Sha(k,E) \not\subset 8\HH^1(k,E)$. 
		\end{Proposition}

		\begin{proof}
			By Lemma~\ref{lem:8tors} we have that $K$, the field of
definition of the $8$-torsion of $E$ is $k(u,v)$, where $u^2 + u = t$, $v^2 + v = (t^2+1)/t$. Then $K/k$ is separable, $\Gal(K/k) \simeq \Z/2\times \Z/2$ and all decomposition groups are cyclic, so $\Sha^1(k,E[8]) \ne 0$ by Theorem~\ref{thm:computeSha1}. This curve has analytic rank $0$ so, by Theorem~\ref{thm:BSD}, $E(k)$ has rank $0$, $\Sha(k,E)$ is finite and the Cassels-Tate pairing is nondegenerate. By Lemma~\ref{lem:torsion} we see that the map $\Sha^1(k,E[8]) \to \Sha(k,E)$ is injective, hence nonzero. The result then follows from Theorem~\ref{thm:characterizedivisibility}.
		\end{proof}	
		
		\begin{Proposition}\label{prop:nonisoH0example}
			The elliptic curve $E:y^2 +xy = x^3 +t^8x^2 + 1/t^8$ over $k = \F_2(t)$ has Mordell-Weil group $E(k) \simeq \Z\oplus \Z/2\Z$ generated by the point $T = (0,1/t^4)$ of order $2$ and the point
	\[
		P = \left( \frac{t^4 + 1}{t^2}, \frac{t^{10} + t^8 + 1}{t^4} \right) 
	\]
	of infinite order. For every $n \ge 3$, the point $2^{n-1}P$ is locally divisible by $2^n$, but is not globally divisible by $2^n$.	
\end{Proposition}

\begin{proof}
	The field of definition of the $8$-torsion of $E$ is $k(u,v), u^2 + u=t,v^2+v=1/t$ by Lemma~\ref{lem:8tors}. All decomposition groups are cyclic, so $\Sha^1(k,E[8]) \simeq \Z/2\Z$ by Theorem~\ref{thm:computeSha1}. The $j$-invariant of $E$ is $t^8$, which is not a $16$-th power. So $E[2^n](\ksep) = E[8](\ksep)$ and, hence, $\Sha^1(k,E[2^n]) \simeq \Sha^1(\Gal(k),E[2^n](\ksep)) \simeq \Z/2\Z$, for all $n \ge 3$. A $2$-descent on $E$ shows that $\Sha(k,E)[2] = 0$ (alternatively, the  analytic rank and algebraic rank are both $1$, so the order of $\Sha(k,E)$ can be computed by the Birch and Swinnerton-Dyer conjectural formula by Theorem~\ref{thm:BSD}). It follows that the nontrivial element of $\Sha^1(k,E[2^n])$ lies in the image of the injective map $\delta_{2^n}:E(k)/2^nE(k) \to \HH^1(k,E[2^n])$. The point $2^{n-1}P$ represents the unique class of order $2$ in $E(k)/2^nE(k)$ that is disjoint from the image of $E(k)_{tors}$. Hence $\delta_{2^n}(2^{n-1}P)$ is the nontrivial element of $\Sha^1(k,E[2^n])$.
\end{proof}
	
		The curves appearing in Propositions~\ref{prop:nonisoH1example} and~\ref{prop:nonisoH0example} are not isotrivial. The next proposition shows that the local-global principle for $2$-divisibility holds for isotrivial elliptic curves over $\F_2(t)$.

	\begin{Proposition}\label{prop:iso}
			If $E$ is an isotrivial elliptic curve over $k = \F_2(t)$, then $\Sha^1(k,E[2^n]) = 0$ for all $n \ge 1$.
	\end{Proposition}
		
		\begin{proof}
			We can assume $E$ is ordinary and not constant. Then $E$ becomes constant after a nonconstant quadratic extension $K/k$ and $K(E[2^n])$ is a constant extension of $K$.  Since $k$ has genus $0$, there is some prime of $k$ that ramified in $K$. Since $K(E[2^n])$ is a constant extension, there is a unique prime of $K(E[2^n])$ above $v$. Since $k(E[2^n]) \subset K(E[2^n])$, there is a unique prime above $v$ in $k(E[2^n])$. Therefore the decomposition group of this prime in $k(E[2^n])$ is isomorphic to $\Gal(k(E[2^n])/k)$. Hence $\Sha^1(k(E[2^n])) = 0$ by Theorem~\ref{thm:computeSha1}.
		\end{proof}
		
		There are, however, isotrivial curves over positive genus global fields of characteristic $2$ for which the local-global principle for divisibility by powers of $2$ can fail.
		
		\begin{Proposition}\label{prop:isoH1example}
			For the isotrivial elliptic curve $E: y^2+xy = x^3 + t^8x^2 + \omega$ over the field $k = \F_4(t,s)$, where $s^2+st=t^3+1, \omega^2+\omega + 1 = 0$ we have $\Sha(k,E) \not\subset 8\HH^1(k,E)$.
		\end{Proposition}

\begin{proof}		
If $c^2+c=t^8$, the change of variables $y=y+cx$ changes the curve
to $y^2+xy = x^3 + \omega$. The $8$-torsion of the latter curve is
defined over $\F_{16}$. So, the $8$-torsion of the former curve is
defined over $K:=k(u,v), u^2 + u=t,v^2+v=\omega$, which is unramified over $k$ and has no inert primes, so Theorem \ref{thm:computeSha1} applies, showing that $\Sha^1(k,E[8])\ne 0$. We show below that $E(k)$ is finite. Then by Lemma~\ref{lem:torsion} the map $\Sha^1(k,E[8]) \to \Sha(k,E)$ is injective and nonzero. By Theorem \ref{thm:BSD}, $\Sha(k,E)$ is finite and the Cassels-Tate pairing on it is nondegenerate. We conclude that $\Sha(k,E) \not\subset 8\HH^1(k,E)$ by Theorem~\ref{thm:characterizedivisibility}.
	
	It remains to show that $E(k)$ is finite.  Both $k$ and $k(u,v)$ are function fields of isogenous elliptic curves and $k$ is the function field of an elliptic curve over ${\bf F}_2$ with $4$ points, so its eigenvalues of Frobenius are $(-1 \pm \sqrt{-7})/2$. On the other hand, $E$ is a twist of $y^2+xy=x^3+ \omega$ which is
an elliptic curve over ${\bf F}_4$ with $6$ points, so its eigenvalues of Frobenius are $(-1 \pm \sqrt{-15})/2$. The eigenvalues of the two curves live in different quadratic fields so they are not isogenous. Therefore $E(k)$ and $E(K)$ are torsion.
\end{proof}

	\begin{Proposition}\label{prop:isoH0example}
		Let $E/k$ be the isotrivial elliptic curve in Proposition~\ref{prop:isoH1example}. For every $n \ge 3$ there is a finite extension $L/k$ for which the local-global principle for divisibility by $2^n$ fails in $E(L)$.
	\end{Proposition}
	
	\begin{proof}
		Using Proposition~\ref{prop:L/k} it suffices to know that $\Sha^1(k,E[2^n]) \ne 0$. This follows from Theorem~\ref{thm:characterizedivisibility}, since Proposition~\ref{prop:isoH1example} shows that $\Sha^1(k,E) \not\subset 2^n\HH^1(k,E)$ for any $n \ge 3$.
	\end{proof}
	
	\section{The proof of Theorem~\ref{thm:characterizedivisibility}}\label{sec:ProofOfTheorem3}
	
		First note that the second statement of the theorem follows from the first since it is known that the Cassels-Tate pairing annihilates only the divisible subgroups \cite[III.9.5]{MilneADT} or \cite[Theorem 1.2]{G-Aduality}. It suffices to prove the first statement assuming $T$ has $p$-primary order.
				
		The sequence
			\begin{equation*}
				0 \to A[p^n] \to A \stackrel{p^n}\to A \to 0
			\end{equation*}	
			is exact on the flat site and gives rise to an exact sequence in flat cohomology. Let $\delta:\HH^1(k,A) \to \HH^2(k,A[p^n])$ be the boundary map arising from~\eqref{eq:KummerSeq}, and let $\iota: \Sha^1(k,A^*[p^n]) \to \Sha(k,A^*)$ be the map induced by the inclusion $A^*[p^n] \hookrightarrow A^*$. One has the  Cassels-Tate pairing,
		\[
			\langle\,,\,\rangle_1 : \Sha(k,A^*)(p) \times \Sha(k,A)(p) \to \Q/\Z
		\]
		(see \cite[Theorem 1.2]{G-Aduality} and \cite[Theorem II.5.6 \& Corollary III.9.5]{MilneADT}) as well as a perfect pairing of finite groups
		\[
			\langle\,,\,\rangle_2 :  \Sha^1(k,A^*[p^n]) \times \Sha^2(k,A[p^n]) \to \Q/\Z
		\]
		(see \cite[Theorem 1.1]{G-Aduality} and \cite[Theorem III.8.2 \& Proposition II.4.13]{MilneADT}). We claim that these pairings are compatible with $\delta$ and $\iota$ in the sense that for any $x \in \Sha^1(k,A^*[p^n])$ and $y \in \Sha(k,A)$,
		\[
			\langle \iota(x), y \rangle_1 = \langle x,\delta(y) \rangle_2\,.
		\]
		The theorem follows from this claim, because $\langle\,,\rangle_2$ is perfect, and $x$ is divisible by $p^n$ in $\HH^1(k,A)$ if and only if $\delta(x) = 0$.
			
		 Let $X$ be the unique smooth complete curve over the field of constants of $k$ having function field $k$. For a sufficiently small open affine subscheme $U \subset X$, we may extend $A$ and $A^*$ to dual abelian schemes $\calA$ and $\calA^*$ over $U$ and, for $0 \le i \le 2$, there are pairings as in the following diagram.
				
	\begin{equation}\label{eq:Pairings}
	\begin{array}[c]{ccccccc}
		(24) & \{\,,\,\}: &D^i(U,\calA^*)(p) & \times & D^{2-i}(U,\calA)(p) & \rightarrow & \Q/\Z \\
		&&\rotatebox{90}{$\hookleftarrow$} && \rotatebox{90}{$\twoheadrightarrow$}&& \rotatebox{90}{$=$}\\
		(13) & \langle \,,\, \rangle:& \HH^i(U,\calA^*)(p) & \times & \HH^{2-i}_c(U,\calA)(p) & \rightarrow & \Q/\Z\\
		&&\rotatebox{90}{$\rightarrow$}\,\theta && \rotatebox{90}{$\leftarrow$}\,\partial_c&& \rotatebox{90}{$\leftarrow$}\\
		(12) &[\,,\,]:&\HH^i(U,\calA^*[p^n]) & \times & \HH^{3-i}_c(U,\calA[p^n]) & \rightarrow & \Q_p/\Z_p\\
		&&\rotatebox{90}{$\hookrightarrow$} && \rotatebox{90}{$\twoheadleftarrow$}&& \rotatebox{90}{$=$}\\
		\text{Lemma 4.7} &(\,,\,):&D^i(U,\calA^*[p^n]) & \times & D^{3-i}(U,\calA[p^n]) & \rightarrow & \Q_p/\Z_p
	\end{array}
	\end{equation}
	(The labels in the left column indicate where the pairing is defined in \cite{G-Aduality}) 
	
	The injective (resp. surjective) maps from (resp. to) the $D^\bullet$ terms are the canonical maps given by the definition of these groups (op. cit. pages 211 and 221). The pairings $\{\,,\}$ and $(\,,\,)$ are induced by the other pairings via these maps. The maps $\theta$ and $\partial_c$ are as in op. cit. Remark 5.3, which shows that they are compatible with the pairings. Therefore \eqref{eq:Pairings} is a commutative diagram of pairings. To prove the claim we use this in the case $i = 1$.
	
	When $U$ is sufficiently small the canonical map $\HH^1(U,\calA^*)(p) \to \HH^1(k,A^*)(p)$ is injective (op. cit. Lemma 6.2), and $D^1(U,\calA^*[p^n]) = \Sha^1(k,A^*[p^n])$ (op. cit. Proposition 4.6). The map $\theta$ is induced by the inclusion $\calA^*[p^n] \hookrightarrow \calA^*$, so the image of $x \in \Sha^1(k,A^*[p^n]) = D^1(U,\calA^*[p^n])$ under the composition
	\[
		D^1(U,\calA^*[p^n]) \to \HH^1(U,\calA^*[p^n]) \stackrel{\theta}\to \HH^1(U,\calA^*) \hookrightarrow \HH^1(k,\calA^*)
	 \]
	 is $\iota(x)$. 
	
	The map $\partial_c$ sits in a commutative diagram with exact rows
	\begin{equation}\label{eq:diag2}
		\xymatrix{
					\HH^1_c(U,\calA) \ar[r]^{p^n}
					&\HH^1_c(U,\calA) \ar[r]^{\partial_c}\ar[d] 
					&\HH^2_c(U,\calA[p^n]) \ar[d]\\
					\HH^1(U,\calA) \ar[r]^{p^n}
					&\HH^1(U,\calA) \ar[r]^{\partial}\ar[dr]
					&\HH^2(U,\calA[p^n])\ar[dr]\\
					&\HH^1(k,A) \ar[r]^{p^n}&\HH^1(k,A) \ar[r]^\delta&\HH^2(k,A[p^n])
		}
	\end{equation}
	(op. cit. proof of Lemma 5.6). The vertical maps in this diagram are the composition of the surjective maps in~\eqref{eq:Pairings} with the canonical inclusions $D^\bullet(U,\star) \subset \HH^\bullet(U,\star)$. For any $U$, the image of $D^2(U,\calA[p^n])$ in $\HH^2(k,\calA[p^n])$ is equal to $\Sha^2(k,A[p^n])$ (proof of op. cit. Proposition 4.5). Lemma 6.5 of op. cit. shows that, provided $U$ is sufficiently small, $D^1(U,\calA)(p)=\Sha(k,A)(p)$. Taken together, this implies that if $U$ is sufficiently small and $y' \in \HH^1_c(U,\calA)$ is a lift of $y \in D^{1}(U,\calA)(p) = \Sha^1(k,\calA)(p)$, then the image of $y'$ under the composition
	\[
		\HH^1_c(U,\calA) \stackrel{\partial_c}\to \HH^2_c(U,\calA[p^n]) \to D^2(U,\calA[p^n]) \to \Sha^2(k,A[p^n])
	\]
	is equal to $\delta(y)$. 
	
	Theorems 4.8 and 6.6 of op. cit. show that $\langle\,,\,\rangle_1$ and $\langle\,,\,\rangle_2$ are given by $\{\,,\}$ and $(\,,\,)$, respectively, provided $U$ is taken to be sufficiently small. Thus we have $\langle \iota(x), y \rangle_1 = \langle x,\delta(y) \rangle_2$ as required. This completes the proof of Theorem~\ref{thm:characterizedivisibility}.

\section*{Acknowledgements}
We would like to thank D. Ulmer for the argument mentioned in Remark \ref{rem:oneplace}, C. Gonz\'alez-Avil\'es for helpful comments including pointing us to \cite[Example C.4.3]{CGP}, and J. Stix for many helpful suggestions on the exposition. The second author would like to also thank the Simons Foundation (Grant \# 234591) for financial support.


\end{document}